\documentclass[11pt]{article}

\usepackage{setspace}
\usepackage{a4wide}
\usepackage{amsthm, amssymb, amstext}
\usepackage[fleqn]{amsmath}
\usepackage{latexsym}
\usepackage{graphicx}
\usepackage{comment}
\usepackage{hyperref}
\usepackage{mathtools}
\usepackage{enumerate} 

\usepackage{todonotes}
\usepackage{comment}

\newtheorem{theorem}{Theorem}[section]
\newtheorem{lemma}{Lemma}[section]

\newtheorem{corollary}{Corollary}[section]

\newtheorem{problem}{Problem}

\theoremstyle{definition}
\newtheorem{definition}{Definition}[section]

\def\inst#1{$^{#1}$}

\title{Reconfiguration Graph for Vertex Colourings of Weakly Chordal Graphs}
\date{}
\author{Carl Feghali\inst{1}\thanks{Email: \texttt{carl.feghali@uib.no}, supported by the Research Council of Norway via the project CLASSIS} \and Ji\v{r}\'{\i} Fiala\inst{2}\thanks{Email: \texttt{fiala@kam.mff.cuni.cz}, supported by the Czech Science Foundation (GA-\v{C}R) project 17-09142S.}}

\begin{document}  

\maketitle
\begin{center}
{\small
\inst{1}
Institutt for informatikk,\\
Universitetet i Bergen, Norway\\
\smallskip

\inst{2}
Department of Applied Mathematics,\\
Charles University, Prague\\
}
\end{center}

\begin{abstract}
The reconfiguration graph $R_k(G)$ of the $k$-colourings of a graph $G$ contains as its vertex set the $k$-colourings of $G$ and two colourings
are joined by an edge if they differ in colour on just one vertex of $G$. 

We show that for each $k \geq 3$ there is a $k$-colourable weakly chordal graph $G$ such that $R_{k+1}(G)$ is disconnected. We also introduce a subclass of $k$-colourable weakly chordal graphs which we call $k$-colourable compact graphs and show that for each $k$-colourable compact graph $G$ on $n$ vertices, $R_{k+1}(G)$ has diameter $O(n^2)$. We show that this class contains all $k$-colourable co-chordal graphs and when $k = 3$ all $3$-colourable $(P_5, \overline{P_5}, C_5)$-free graphs. We also mention some open problems. 
\end{abstract}

\section{Introduction}

Let $G$ be a graph, and let $k$ be a non-negative integer. 
A $k$-colouring of $G$ is a function $f: V(G) \rightarrow \{1, \dots, k\}$ such that $f(u) \not= f(v)$ whenever $(u, v) \in E(G)$. 
The reconfiguration graph $R_k(G)$ of the $k$-colourings of $G$ has as vertex set the set of all $k$-colourings of $G$ and two vertices of $R_k(G)$ are adjacent if they differ on the colour of exactly one vertex (the change of the colour is the so called \emph{colour switch}). For a positive integer $\ell$, the \emph{$\ell$-colour diameter} of a graph $G$ is the diameter of $R_{\ell}(G)$. 

In the area of reconfigurations for colourings of graphs, one focus is to determine the complexity of deciding whether two  given colourings of a graph can be transformed into one another by a sequence of recolourings (that is, to decide whether there is a path between these two colourings in the reconfiguration graph); see, for example, \cite{3colouring, johnson2, brewster, bonsma}. Another focus is to determine the diameter of the reconfiguration graph in case it is connected or the diameter of its components if it is disconnected \cite{bonamy, johnson1, bonamy13, bousquet, feghali2}. We refer the reader to \cite{heuvel, nishimura} for excellent surveys on reconfiguration problems.

 In this note, we continue the latter line of study of reconfiguration problems. In Section \ref{sec:weakly}, we shall show that the $(k + 1)$-colour diameter of $k$-colourable weakly chordal graphs can be infinite.  On the positive side, in Section \ref{sec:quadratic}, we shall consider two specific subclasses of $k$-colourable perfect graphs and show that their $(k + 1)$-colour diameter is quadratic in the order of the graph.

\section{Preliminaries}

For a graph $G = (V, E)$ and a vertex $u \in V$, let $N_G(u) = \{v: uv \in E\}$. A \emph{separator} of a graph $G = (V,E)$ is a set $S \subset V$
such that $G -  S$ has more connected components than $G$. If two vertices $u$ and $v$
that belong to the same connected component in $G$ are in two different connected
components of $G - S$, then we say that $S$ \emph{separates} $u$ and $v$. A \emph{chordless path} $P_n$ of length $n - 1$ is the graph with vertices $v_1, \dots, v_n$ and edges $v_iv_{i+1}$ for
$i = 1, \dots,n -1$. It is a \emph{cycle} $C_n$ of length $n$ if the edge $v_1v_n$ is also present.  

The \emph{complement} of $G$ is denoted $\overline{G} = (V, \overline{E})$. It is the graph on the same vertex set as $G$ and there is an edge in $G$ between two vertices $u$ and $v$ if and only if there is no edge between $u$ and $v$ in $\overline{G}$. A set of vertices in a graph is \emph{anticonnected} if it induces a graph whose complement is connected. A \emph{clique} or a \emph{complete graph} is a graph where every pair
of vertices is joined by an edge. The size of a largest clique in a graph $G$ is denoted $\omega(G)$. The \emph{chromatic number} $\chi(G)$ of a graph $G$ is the least integer $k$ such that $G$ is $k$-colourable. 

A graph $G$ is \emph{perfect} if $\omega(G') = \chi(G')$ for every (not necessarily proper) subgraph $G'$ of $G$.  A \emph{hole} in a graph is a cycle of length at least 5 and an \emph{antihole} is the complement of a hole. A graph is perfect if it is (odd hole, odd antihole)-free \cite{chudnovsky}. A graph is \emph{weakly chordal} if it is (hole, antihole)-free. A graph is \emph{co-chordal} if it is ($\overline{C_4}$, anti-hole)-free. Every weakly chordal graph is perfect. Every co-chordal graph and every $(P_5, \overline{P_5}, C_5)$-free graph is weakly chordal. 

A \emph{2-pair} of a graph $G$ is a pair $\{x, y\}$ of nonadjacent distinct vertices of $G$ such that every chordless path from $x$ to $y$ has length 2. We often use the following well-known lemma: 

\begin{lemma}[Hayward et al. \cite{hayward}]\label{lem:hayward}
A graph $G$ is weakly chordal graph if and only if every subgraph of $G$ is either a complete graph or it contains a 2-pair. 
\end{lemma} 

\section{Weakly chordal graphs}\label{sec:weakly}

%In this section, we consider a question from \cite{bonamy} which asks whether the $(k+1)$-colour diameter of $k$-colourable perfect graphs is quadratic. We answer this question in the negative in the following theorem.

In this section, we establish the following result. 

\begin{theorem}\label{thm:weakly}
For each $k \geq 3$ there exists a $k$-colourable weakly chordal graph $G_k$ such that $R_{k + 1}(G_k)$ is disconnected. 
\end{theorem}

The graph $G_k$ is depicted in Figure~\ref{fig:counterexample}.

\begin{figure}
\begin{center}
\scalebox{1.2}{\includegraphics[page=2]{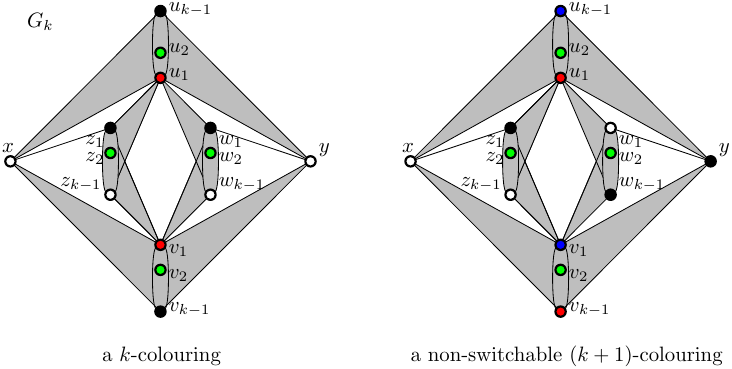}}
\end{center}
\caption{The graph $G_k$. Each gray area corresponds to a clique.}
\label{fig:counterexample}
\end{figure}

In other words, Theorem \ref{thm:weakly} states that for each $k \geq 3$ the $(k+1)$-colour diameter of $k$-colourable weakly chordal graphs can be infinite.  %Since the class of weakly chordal graphs is a proper subclass of the class of perfect graphs, Theorem \ref{thm:weakly} answers their question in a strong form. 
It is worth mentioning that the case $k = 2$ is already known \cite{bonamy} as the class of 2-colourable weakly chordal graphs is precisely the class of chordal bipartite graphs. %CF added next sentence
It is also worth mentioning that Bonamy, Johnson, Lignos, Patel and Paulusma \cite{bonamy} asked whether the $(k + 1)$-colour diameter of $k$-colourable perfect graphs is connected. This was answered negatively in \cite{treewidth} -- the counterexample consists of a complete bipartite graph minus a matching. Our Theorem \ref{thm:weakly} thus strengthens this counterexample.

\begin{proof}[Proof of Theorem \ref{thm:weakly}]
It suffices to construct for each $k \geq 3$ a $k$-colourable weakly chordal graph $G_k$ and a $(k+1)$-colouring of $G$ such that each of the $k + 1$ colours appear in the closed neighbourhood of every vertex of $G_k$, as then no vertex of $G_k$ can get recoloured. 

Such graph is depicted in Figure~\ref{fig:counterexample}. It is formed from the disjoint union of four complete graphs $K_{k-1}$, one on vertices $u_i$ for $i\in \{1,\dots,k-1\}$, the other three on vertices $v_i,w_i,z_i$, respectively, and two further vertices $x$ and $y$. These parts are joined together by additional edges such that $x$ and $y$ are connected to each $u_i$ and to each $v_i$; $u_1$ and $v_1$ are connected to each $z_i$ and to each $w_i$; and finally, $G_k$ contains two further edges $xz_1$ and $yw_1$.

A possible $k$-colouring of $G_k$ is schematically shown on the left side of Figure~\ref{fig:counterexample}; in both pictures each 4-tuple $u_i,v_i,w_i,z_i$ for $i\in\{2,\dots,k-2\}$ receives a unique colour. 
On the other hand, in the $(k+1)$-colouring depicted on the right, every vertex of $G_k$ has its neighbours coloured by the $k$ remaining colours, hence no vertex can be recoloured.
Hence, this colouring corresponds to an isolated vertex in the reconfiguration graph, and thus the reconfiguration graph $R_{k+1}(G_k)$ is disconnected.

It remains to show that $G_k$ is weakly chordal. Observe first that any distinct $u_i,u_j$ with $i,j\in \{2,\dots,k-1\}$ have the same neighbourhood, hence no hole and no antihole may contain both of them. The same holds for vertices $v_i,w_i,z_i$, respectively. Hence without loss of generality we may assume that no vertex with index at least three participates in a hole nor in an antihole. In other words, it suffices to restrict ourselves only to the graph $G_3$ to show that it is weakly chordal.

By examining possible paths, one can also realise that vertices $x,y$ as well as $u_1,v_1$ form a 2-pair.
Since no hole may contain a 2-pair, we may assume without loss of generality that a possible hole does not contain $y$ and $v_1$. The vertex $x$ separates $v_2$ and also the vertex $u_1$ separates $w_1,w_2$ from the rest in the graph $G_3 - \{y, v_1\}$, hence $v_2,w_1,w_2$ also do not belong to a hole. We were left with five vertices $x,u_1,u_2,z_1,z_2$ which do not induce a hole, hence $G_3$ does not contain a hole at all.

Now assume for a contradiction that $G_3$ contains an antihole on at least 6 vertices. The graph $G_3$ contains only 6 vertices of degree at least 4, hence the antihole contains exactly 6 vertices. The neighbourhood of $u_2$ induces a diamond (a $K_4$ minus an edge), hence it does not belong to the antihole as no diamond is an induced subgraph of $\overline{C_6}$. The same holds for $v_2,z_2,w_2$. We are left with vertices $x,y,u_1,v_1,w_1,z_1$ which do not induce an antihole in $G_3$, hence $G_3$ does not contain an antihole at all. Since $G_3$ is weakly chordal, we have shown that $G_k$ is also weakly chordal for each~$k\ge 4$.
\end{proof}

\section{Quadratic diameter}\label{sec:quadratic}

In this section, we introduce a subclass of $k$-colourable weakly chordal graphs that we call $k$-colourable {compact graphs}.  We show  in Theorem \ref{thm:strong}  that for each $k$-colourable compact graph $G$ on $n$ vertices the diameter of $R_{k + 1}(G)$ is $O(n^2)$. We then show in Lemma \ref{lem:co-chordal} that $k$-colourable co-chordal graphs are $k$-colourable compact and in Lemma \ref{lem:p5} that $3$-colourable $(P_5, \overline{P_5}, C_5)$-free graphs are $3$-colourable compact. 

For a 2-pair $\{u , v\}$ of a weakly chordal graph $G$, let $S(u, v) = N_G(u) \cap N_G(v)$. Note that, by the definition of a 2-pair, $S(u, v)$ is a separator of $G$ that separates $u$ and $v$. Let $C_v$ denote the component of $G \setminus S(u, v)$ that contains the vertex $v$. 

\begin{definition}\label{def:colourcompact}
%For a positive integer $k$, 
A  weakly chordal graph $G$ is said to be \emph{compact} if every subgraph $H$ of $G$ either
\begin{itemize}
\item[(i)]  is a complete graph,  or 
\item[(ii)] contains a 2-pair $\{x, y\}$ such that $N_H(x) \subseteq N_H(y)$, or
\item[(iii)] contains a 2-pair $\{x, y\}$ such that $C_x \cup S(x, y)$ induces a clique on at most three vertices.    
\end{itemize}
\end{definition}

\begin{theorem}\label{thm:strong}
Let $k$ be a positive integer, and let $G$ be a $k$-colourable compact graph on $n$ vertices. Then $R_{k+1}(G)$ has diameter $O(n^2)$. 
\end{theorem}

\begin{proof}
It suffices to show that we can recolour a $(k+1)$-colouring $\alpha$ of $G$ to a $(k+1)$-colouring $\beta$ by recolouring every vertex at most $2n$ times. 

We first suppose that $G$ is a complete graph. In this case, we know from \cite{bonamy} that we can recolour $\alpha$ to $\beta$ by recolouring every vertex at most $2n$ times. We now consider the case when $G$ is not a complete graph but satisfies condition (ii) of compact graphs. We use induction on the number of vertices of $G$. Let $\{x, y\}$ be a 2-pair of $G$ such that $N_G(x) \subseteq N_G(y)$. From $\alpha$ and $\beta$, we can immediately recolour $x$ with, respectively, $\alpha(y)$ and $\beta(y)$. Let $G' = G - \{x\}$, and let $\alpha_{G'}$ and $\beta_{G'}$ denote the restrictions of $\alpha$ and $\beta$ to $G'$. By our induction hypothesis, we can transform $\alpha_{G'}$ to $\beta_{G'}$ by recolouring every vertex at most $2(n - 1)$ times. We can extend this sequence of recolourings to a sequence of recolourings in $G$ by recolouring $x$ using the same recolouring as $y$. Then $x$ gets recoloured as many times as $y$ as needed. 

\begin{figure}
\begin{center}
\scalebox{1.2}{\includegraphics{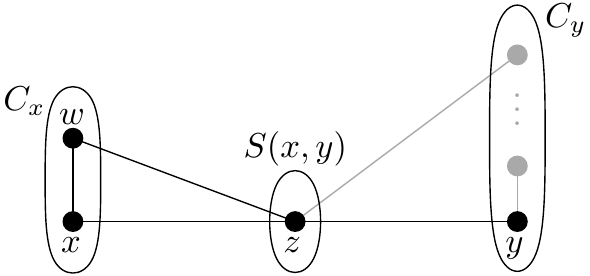}}
\end{center}
\caption{For the proof of Theorem~\ref{thm:strong}. The component $C_y$ contains further vertices and edges outlined in gray.}
\label{fig:quadratic}
\end{figure}

Suppose that $G$ satisfies condition (iii) of compact graphs. We use induction on the number of vertices. If $S(x, y)$ contains exactly two vertices, then $C_x = \{x\}$ and hence $G$ satisfies condition (ii) of compact graphs.  So we can assume that $S(x, y)$ is a single vertex $z$ and $C_x$ consists of $x$ and another vertex $w$, see Figure~\ref{fig:quadratic}. From $\alpha$ and $\beta$, we can recolour $x$ and $w$ to another colour either immediately or by first recolouring $w$ and $x$, respectively. Let $G^* = G - \{x, w\}$. By our induction hypothesis, we can transform $\alpha_{G^*}$ to $\beta_{G^*}$ by recolouring every vertex at most $2(n - 2)$ times. We can extend this sequence of recolourings to a sequence of recolourings in $G$ by recolouring $x$ and $w$ whenever $z$ is recoloured to their colour. At the end of the sequence we recolour $x$ and $w$ so that they agree in both colourings. As $x$ and $w$ are recoloured at most two more times as $z$, this completes the proof. \end{proof}

\begin{lemma}\label{lem:co-chordal}
Every $k$-colourable co-chordal graph is compact. 
\end{lemma}

\begin{proof}
Let $G$ be a $k$-colourable co-chordal graph. If $G$ is a complete graph, then $G$ is compact by definition. Otherwise, since $G$ is weakly chordal, $G$ contains a 2-pair $\{x, y\}$ by Lemma \ref{lem:hayward}. If $x$ has a neighbour $x_1$ that is not a neighbour of $y$ and $y$ has a neighbour $y_1$ that is not a neighbour of $x$, then $x_1$ is not adjacent to $y_1$, as otherwise $S(x, y)$ does not separate $x$ and $y$. But then the edges $xx_1$ and $yy_1$ form $\overline{C_4}$, a contradiction. Therefore, $N_G(x) \subseteq N_G(y)$ or vice-versa and hence the graph $G$ is compact as required.  
\end{proof}

\begin{lemma}\label{lem:p5}
Every $3$-colourable $(P_5, \overline{P_5}, C_5)$-free graph is compact. 
\end{lemma}

The proof of this lemma will require a little more work. First, we need some definitions and auxiliary results. When $T$ is a set of vertices of a graph $G$, a set $D \subseteq V (G) \setminus T$ is $T$-complete if each vertex of $D$ is adjacent to each vertex of $T$. Let $D(T)$ denote the set of all $T$-complete vertices. 

\begin{lemma}[Trotignon and Vu{\v{s}}kovi{\'c} \cite{trotignon}]\label{lem:ct}
Let $G$ be a weakly chordal graph, and let $T \subseteq V(G)$ be a set of vertices such that $G[T]$ is anticonnected and $D(T)$ contains at least two non-adjacent vertices. If $T$ is inclusion-wise maximal with respect to these properties, then any chordless path of $G \setminus T$ whose ends are in $D(T)$ has all its vertices in $D(T)$. 
\end{lemma}

The following corollary is implicit in \cite{trotignon}. 

\begin{corollary}\label{cor:p5}
Let $G$ be a weakly chordal graph that contains a chordless path $P$ of length $2$. Then there exists an anticonnected set $T$ containing the centre of $P$, such that $D(T)$ contains a 2-pair of $G$. 
\end{corollary}

In particular, the 2-pair can always be found in the neighbourhood of the centre of $P$. 

\begin{proof}
We start with the centre of $P$ to build our set $T$ as in Lemma \ref{lem:ct}. Then $D(T)$ is not a clique as it contains both ends of $P$. Hence, by definition of weakly chordal graphs, $D(T)$ contains a 2-pair. This 2-pair is also a 2-pair of $G$ by Lemma \ref{lem:ct}. 
\end{proof}

We are now ready to prove Lemma \ref{lem:p5}. 

\begin{proof}[Proof of Lemma \ref{lem:p5}]
Let $G$ be a 3-colourable $(P_5, \overline{P_5}, C_5)$-free graph.
Suppose towards a contradiction that $G$ is not compact. 
In particular, $G$ is not a complete graph, as otherwise $G$ would be compact by the definition. 
Since $G$ is weakly chordal, it contains a 2-pair $\{x, y\}$.  
We choose $\{x , y\}$ such that $|V(C_x)|$ is minimum over all 2-pairs $\{x, y\}$ of $G$.
  
Denote by $G'$ the subgraph of $G$ induced by the union of $S(x,y)$ and the vertices of $C_x$. 
The subgraph $G'$ is not complete, as $G$ would be compact, 
hence $G'$ contains a chordless path of length 2. Let us argue that $G'$ contains a chordless path of length 2 whose centre is, in fact, a member of $C_x$. 

If $C_x$ is not complete, then this is immediate. And if $S(x,y)$ is not complete, then it contains a pair of vertices $u$ and $v$ that are not adjacent, so we take $u$, $x$, $v$ to be our path. Hence we can assume that $C_x$ and $S(x, y)$ are both complete and, as $G'$ is not complete, there must be a vertex $u$ of $C_x$ and a vertex $v$ of $S(x, y)$ such that $u v \not\in E(G)$. Then we can take $u$, $x$, $v$ as our path and our aim is achieved. 

Applying Corollary \ref{cor:p5} with $P$ being a chordless path of length 2 whose centre is in $C_x$, we find that $G'$ contains a 2-pair $\{z, w\}$ that is also a 2-pair of $G$.

We next want to argue that $z, w\in S(x, y)$. For a contradiction, assume without loss of generality that $z$ belongs to $C_x$. This implies that $S(z, w) \subseteq V(G')$. 
So there must be a component $C_1$ of $G \setminus S(z, w)$ such that $C_1$ and $S(x, y)$ do not have a vertex in common since $y$ is adjacent to every vertex of $S(x, y)$. Therefore, we find that $C_1 \subseteq C_x$. If $C_1 = C_x$, then $S(x, y) = S(z, w)$ and hence $z, w \in C_1$ which is impossible because $S(z, w)$ separates $z$ and $w$. Therefore, $|V(C_1)| < |V(C_x)|$ holds, which contradicts our choice of $C_x$.

\begin{figure}
\begin{center}
\scalebox{1.2}{\includegraphics{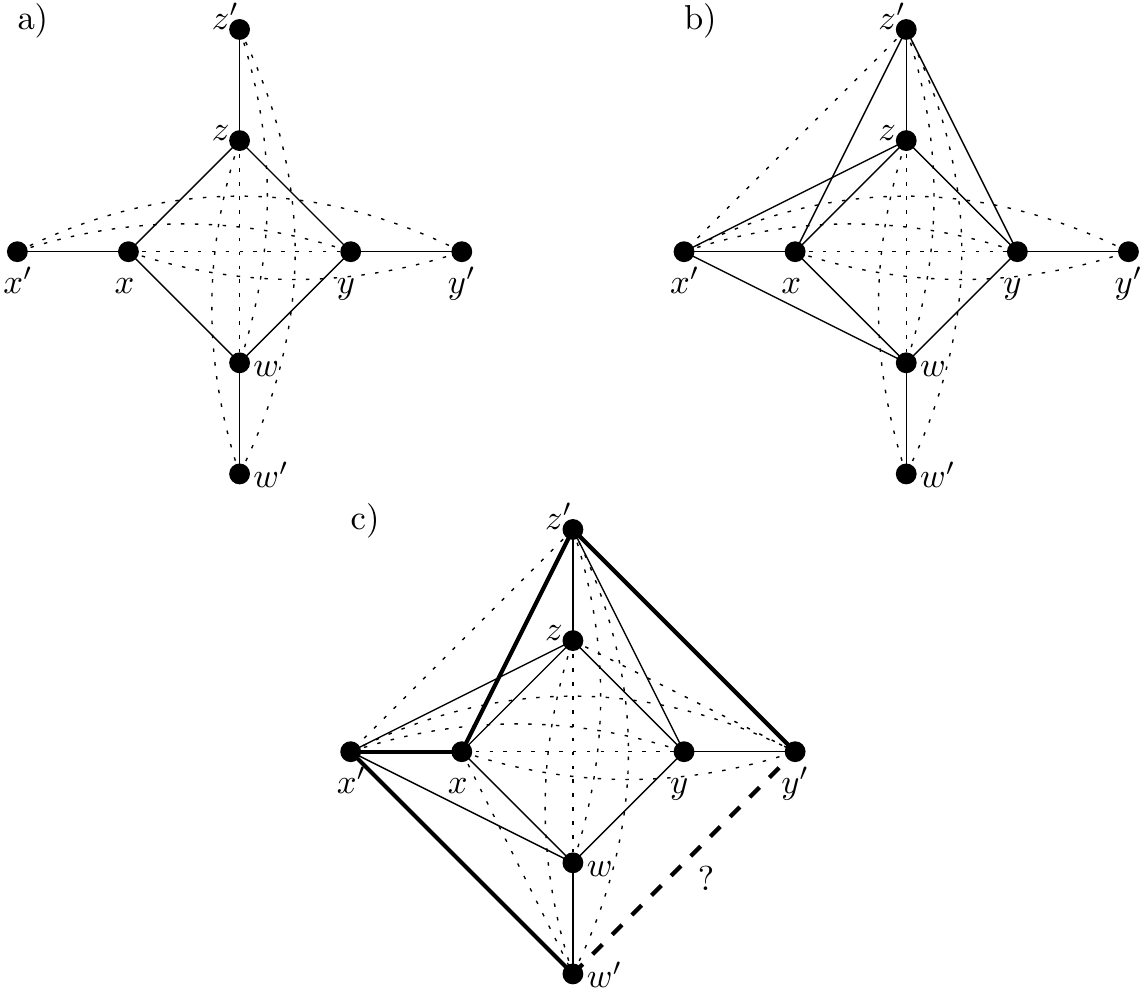}}
\end{center}
\caption{The case analysis for the proof of Lemma~\ref{lem:p5}. The dotted connections indicate nonadjacent vertices.}
\label{fig:windmill}
\end{figure}

Hence we have concluded that $z, w\in S(x, y)$. Now the vertices $x, z, y, w$ form a cycle such that $\{x, y\}$ and $\{z, w\}$ are 2-pairs. Since $G$ is not compact, $N_G(x) \not\subseteq N_G(y)$, hence there exists 
a vertex $x'$ that is adjacent to $x$ but not to $y$.
Analogously, there are vertices $y', z', w'$ such that  $yy',zz',ww' \in E$, but $xy',wz',zw' \notin E$.  If $x' = z'$, then $x'$ must be adjacent to $y$ or $w$ else $x', x, z, y, w$ form $\overline{P_5}$. But $x'$ is not adjacent to $y$ and $z'$ not adjacent to $w$, thus $x' \not = z'$. Similarly, $y' \not = w'$ and hence $x', y', w', z'$ are distinct.
Moreover, since $S(x,y)$ is a separator, we get $x'y'\notin E$ and analogously $w'z'\notin E$, see Figure~\ref{fig:windmill} a).  

If both $z'$ and $w'$ are not adjacent to $x$, then $z', z, x, w, w'$ form $P_5$. So we can assume without loss of generality that $z'$ is adjacent to $x$. If $z'$ is not adjacent to $y$, then $z', z, x, w, y$ would form $\overline{P_5}$ as $w$ is not adjacent to $z'$. Thus $z'$ is adjacent to both $x$ and $y$. 

Similarly, to avoid $P_5$ on vertices $x',x,z,y,y'$ either $x'$ or $y'$ must be adjacent to $z$ and hence also to $w$, so we assume without loss of generality that $x'$ is adjacent to $z$ and $w$.
If $x'$ is adjacent to $z'$, then $x, z, z', x'$ form $K_4$, a contradiction with the assumption the $G$ is 3-colourable, see Figure~\ref{fig:windmill} b).

To avoid a $P_5$ on the vertices $x',x, z', y, y'$, the vertices $z'$ and $y'$ are forced to be adjacent.
Similarly $x'w'\in E$ as otherwise $z',z, x', w, w'$ induce a $P_5$.
To avoid a $K_4$ on the vertices $z, z', y, y'$, the vertices $z$ and $y'$ need to be nonadjacent.
Similarly $xw'\notin E$ as otherwise $x, x', w, w'$ induce a $K_4$.

Now the vertices $w',x',x,z',y'$ induce either a $C_5$ or a $P_5$, depending whether the edge $y'w'$ is present or not, see Figure~\ref{fig:windmill} c). In either case we arrive at a contradiction and the lemma is proved.  
\end{proof}

We are aware the concept of compact graphs does not fit tight with the class of $(P_5, \overline{P_5}, C_5)$-free graphs, as some of these graphs need not to be $k$-colourable compact graphs for $k\ge 4$. 
An example of such graph $H$ for $k=4$ is depicted in Figure.~\ref{fig:counterexample2}.

\begin{figure}
\begin{center}
\scalebox{1.2}{\includegraphics{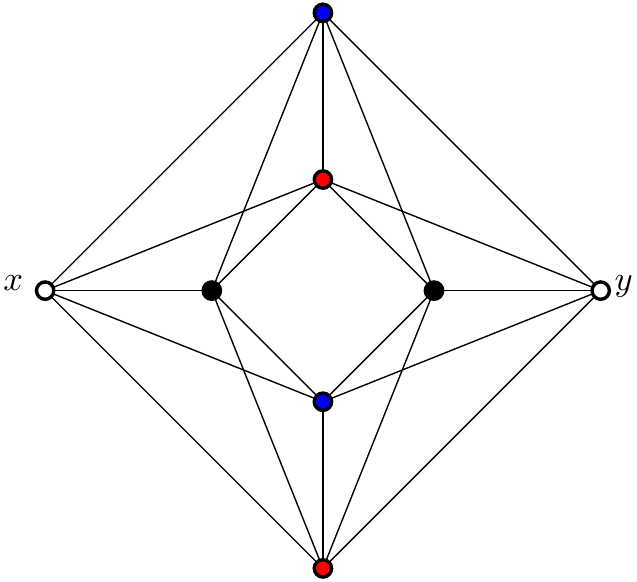}}
\end{center}
\caption{A $(P_5, \overline{P_5}, C_5)$-free 4-colourable graph $H$ that is not compact.}
\label{fig:counterexample2}
\end{figure} 

Due to symmetries of the graph $H$ it suffices without loss of generality to consider only the 2-pair $\{x,y\}$ as other 2-pairs could be mapped onto $\{x,y\}$ by an automorphism of $H$. Observe that this 2-pair violates the conditions of the definition~\ref{def:colourcompact} for $H$ to be 4-colourable compact.

Any choice of five vertices from $H$ would contain two vertices joined by a horizontal or a vertical edge, and such edge cannot be extended to an induced $P_3$, hence $H$ is also $P_5$-free.
Also, such choice of five vertices would contain two opposite vertices either of the inner $C_4$ or from the outer one, like the vertices $x$ and $y$. As such two vertices form an $2$-pair, $H$ contains no $C_5$. 
Finally, $H$ has only two induced $C_4$ and neither could be completed by any fifth vertex to a $\overline{P_5}$.

\section{Concluding remarks}

 We end this note with two open problems. 
 
 \begin{problem}
For which integer $\ell > k + 1$  is the $\ell$-colour diameter of $k$-colourable weakly chordal graphs connected? 
 \end{problem}
 
 %\begin{problem}
 %For which integer $\ell > k + 1$  is the $\ell$-colour diameter of $k$-colourable perfect graphs quadratic? 
 %\end{problem}
 
 \begin{problem}
 Is it true that the $(k+1)$-colour diameter of $k$-colourable $(P_5, \overline{P_5}, C_5)$-free graphs is quadratic for each $k \geq 4$? 
 \end{problem}

 \bibliography{bibliography}{}
\bibliographystyle{abbrv}
 
\end{document}